\documentclass[11pt]{amsart}

\usepackage{amssymb,amsthm,amsfonts,latexsym,mathtools}

\usepackage{amsmath}
\usepackage{mathrsfs}
\usepackage{stmaryrd}
\usepackage{color}
\usepackage{enumitem}
\usepackage{tabularx}
\usepackage{graphicx}
\usepackage{accents}

\usepackage[backref]{hyperref}
\usepackage{tikz}
\usepackage{tikz-cd}

\usepackage{pbox}
\usepackage{comment}

\input{xy}
\xyoption{all}
\xyoption{poly}
\usepackage[all]{xy}
\allowdisplaybreaks

 \def\B{{\mathcal B}}
 \def\C{{\mathcal C}}

 \def\L{{\mathcal L}}

 \def\P{{\mathcal P}}

 \def\U{{\mathcal U}}

 \DeclareMathOperator\Lk{Lk}

 \DeclareMathOperator\GL{GL}

 \DeclareMathOperator{\rk}{rk}
 \newcommand{\tq}{\mathrel{{\ensuremath{\: : \: }}}}

  \makeatletter
  \newcommand{\numberlike}[2]{%
     \expandafter\def\csname c@#1\endcsname{%
         \expandafter\csname c@#2\endcsname}%
  }
  \makeatother

  \def\DefaultNumberTheoremWithin{section}

  \theoremstyle{plain}
  \newtheorem{lemma}{Lemma}
     \numberwithin{lemma}{\DefaultNumberTheoremWithin}
     \labelformat{lemma}{Lemma~#1}
  
     \numberwithin{claim}{\DefaultNumberTheoremWithin}
     \numberlike{claim}{lemma}
     \labelformat{claim}{Claim~#1}
  \newtheorem{theorem}{Theorem}
     \numberwithin{theorem}{\DefaultNumberTheoremWithin}
     \numberlike{theorem}{lemma}
     \labelformat{theorem}{Theorem~#1}
  \newtheorem{corollary}{Corollary}
     \numberwithin{corollary}{\DefaultNumberTheoremWithin}
     \numberlike{corollary}{lemma}
     \labelformat{corollary}{Corollary~#1}
     
  \newtheorem{proposition}{Proposition}
     \numberwithin{proposition}{\DefaultNumberTheoremWithin}
     \numberlike{proposition}{lemma}
     \labelformat{proposition}{Proposition~#1}

     \numberwithin{conjecture}{\DefaultNumberTheoremWithin}
     \numberlike{conjecture}{lemma}
     \labelformat{conjecture}{Conjecture~#1}

  \newtheorem*{warning*}{Warning}

  \theoremstyle{definition}
  \newtheorem{definition}{Definition}
     \numberwithin{definition}{\DefaultNumberTheoremWithin}
     \numberlike{definition}{lemma}
     \labelformat{definition}{Definition~#1}

  \theoremstyle{definition}
  \newtheorem{question}{Question}
     \numberwithin{question}{\DefaultNumberTheoremWithin}
     \numberlike{question}{lemma}
     \labelformat{question}{Question~#1}

  \theoremstyle{definition}
  
     \numberwithin{problem}{\DefaultNumberTheoremWithin}
     \numberlike{problem}{lemma}
     \labelformat{problem}{Problem~#1}

  \theoremstyle{remark}
  \newtheorem{remark}{Remark}
     \numberwithin{remark}{\DefaultNumberTheoremWithin}
     \numberlike{remark}{lemma}
     \labelformat{remark}{Remark~#1}
  \theoremstyle{remark}

  \newtheorem{example}{Example}
     \numberwithin{example}{\DefaultNumberTheoremWithin}
     \numberlike{example}{lemma}
     \labelformat{example}{Example~#1}

\labelformat{equation}{(#1)}
\labelformat{figure}{Figure~#1}
\labelformat{table}{Table~#1}
\labelformat{chapter}{Chapter~#1}
\labelformat{appendix}{Appendix~#1}
\labelformat{section}{Section~#1}
\labelformat{subsection}{Subsection~#1}

\newlength{\dhatheight}

\newcommand{\red}[1]{#1^*}
\newcommand{\redm}[1]{#1^\circ}

\newcommand{\Ord}{\mathcal{O}}

\newcommand{\cl}{\mathrm{Cl}}
\newcommand{\fin}{fin}

\newcommand*{\doi}[1]{\href{http://dx.doi.org/#1}{doi: #1}}
\begin{document}

\title[Notes on the topology of independence structures]{Notes on the topology of independence structures}

\author{Kevin I. Piterman}
\address{Vrije Universiteit Brussel\\
Department of Mathematics \& Data Science\\
1050 Brussels, Belgium}
\email{kevin.piterman@vub.be}

\author{Volkmar Welker}
\address{Philipps-Universit\"at Marburg \\
Fachbereich Mathematik und Informatik \\
35032 Marburg, Germany}
\email{welker@mathematik.uni-marburg.de}

\begin{abstract}
Following Welsh, a pre-independence space (pi-space) is a set $M$ together with a non-empty collection $I(M)$ of subsets of $M$, called independent sets, which is closed under taking subsets, and finite independent sets satisfy the exchange property from matroid theory.

We show that $I(M)$, viewed as a poset, is contractible if it is infinite-dimensional, and Cohen-Macaulay otherwise.
Moreover, the proper part of the associated poset of flats is also contractible in the infinite-dimensional case, and Cohen-Macaulay otherwise.

These results generalize those for independence complexes and geometric lattices of (finite) matroids.

\end{abstract}

\maketitle

\section{Introduction}

For a matroid $M$, the homological and topological properties of the simplicial complex $I(M)$ of its independent sets  
and the order complex of the proper part of its lattice of flats $\L(M)$ 
have been studied intensively in combinatorics (see e.g., \cite{Bjorner1} for a survey or \cite{ADH} for recent uses). Both are homotopy equivalent to a wedge of spheres. Moreover, they are shellable and hence homotopically Cohen-Macaulay. 
In addition, the topology of the extended Bergman complex of $M$ -- a combination of both $I(M)$ and $\L(M)$ -- has turned into the focus (see \cite{BHMPW, BKR}). In all those studies, matroids are finite structures. Indeed, classically, matroid theory is confined to the finite setting.

The question of how to extend the matroid 
axioms to the infinite setting while 
preserving its coveted properties has been the subject of intensive research
(see, for example, \cite{infM, Oxley,Welsh}). In these notes, we set up a context of ``infinite matroids" $M$ that recovers the topological properties of the independence complex $I(M)$ and of the lattice of flats $\L(M)$.

Further motivation comes from the case when $M$ is the projective geometry associated to a finite-dimensional vector space $V$ over a not-necessarily-finite field. Then  -- using suitable definitions -- the lattice of flats
is the poset of linear subspaces of $V$. The order complex of its proper
part is the building associated to $\GL(V)$ and hence has the homotopy type of a wedge of spheres (see \cite{G,S}).
The fact that in this case $I(M)$ is Cohen-Macaulay too, can for example be deduced from \cite{vdk}. There, more generally, Cohen-Macaulayness is shown for the complex
of partial bases or unimodular sequences for finite-rank free modules over certain Dedekind domains.

Even in this classical setting of vector spaces, the case of infinite-dimensional
vector spaces seems to be unexplored.

Our approach to ``infinite matroids" follows ideas by Welsh
\cite[Chapter 20]{Welsh} where he defines pre-independence spaces, pi-spaces for short,
over a (possibly finite set) $M$ through a set $I(M)$ of
independent sets (see \ref{sec:pi-spaces} for definitions).
This concept includes the case when $M$ is a finite-dimensional
vector space over an arbitrary field and $I(M)$ the set of its linearly
independent subsets.



For infinite $M$, the set system $I(M)$ may no longer be a simplicial complex, though. For that reason, we consider $I(M)$ as a poset with the order given by inclusion.
Via the order complex of $I(M) \setminus \{\emptyset\}$, we can still regard $I(M)$ as a topological space.
Note that we will always have $\emptyset \in I(M)$. Thus, in order
to capture the essential part of the topology of $I(M)$, 
the empty set must be excluded. 
In case every independent set in $M$ is finite, $I(M)$ is a simplicial complex and the order complex of $I(M) \setminus \{\emptyset\}$ 
is the barycentric subdivision of $I(M)$. 
When $I(M)$ is finite-dimensional, then it is even a pure complex.

The following theorem describes the homotopy type of $I(M)$.

\begin{theorem}
\label{thm:homotopyIndependentSetsPoset}
Let $(M,I(M))$ be a pi-space.
\begin{enumerate}
    \item If $I(M)$ is finite-dimensional, then $I(M)$ is a homotopically Cohen-Macaulay simplicial complex.
    \item If $I(M)$ is infinite-dimensional, then $I(M)$ is a contractible poset.
\end{enumerate}
\end{theorem}

As a consequence, we obtain analogous results on the homotopy type of the poset of injective words on $I(M)$ (see \ref{coro:injectiveWordsHomotopy}).
In particular, this shows that the poset of unimodular sequences of a vector space is Cohen-Macaulay in the finite-dimensional case, and it is contractible otherwise. Except for Cohen-Macaulayness this 
also follows from \cite[p. 274]{vdk}.

The second important geometric structure associated with a matroid is its lattice of flats.
Recall that a flat of a matroid $M$ is a subset $F\subseteq M$ such that for any $x\in M\setminus F$ and independent set $\sigma\subseteq F$, $\sigma\cup \{x\}$ is an independent set.
For pi-spaces, we consider the same definition, and denote by $\L(M,I(M))$ the poset of all flats with order given by inclusion.
Note that $M$ is the unique maximal element of this poset.
Furthermore, the set $0_{I(M)} := \big\{\,x\in M\, \tq\, \{x\}\notin I(M)\,\big\}$ is the unique minimal element of $\L(M,I(M))$.
Thus, $\L(M,I(M))$ is a bounded poset,
and when we speak about its topology, we mean the topology of the order complex of
its proper part $\red{\L(M,I(M))} = \L(M,I(M)) \setminus \{M,0_{I(M)}\}$.

We have the following theorem.

\begin{theorem}
\label{thm:homotopyPosetOfFlats}
Let $(M,I(M))$ be a pi-space.
Then $\L(M,I(M))$ is a complete lattice and the following hold:
\begin{enumerate}
\item If $I(M)$ is finite-dimensional, then $\L(M,I(M))$ is homotopically Cohen-Macaulay.
\item If $I(M)$ is infinite-dimensional, then $\red{\L(M,I(M))}$ is contractible.
\end{enumerate}
\end{theorem}

In particular, \ref{thm:homotopyIndependentSetsPoset} and \ref{thm:homotopyPosetOfFlats} describe the homotopy type of the poset of linearly independent subsets and the poset of non-zero proper subspaces of vector spaces of arbitrary dimension.

We close this introduction with a few remarks.

First, our definition of flats does not coincide with that given by Welsh, and the bounded poset $\L(M,I(M))$ may not satisfy the combinatorial properties of a geometric lattice.
Indeed, \cite[p. 388]{Welsh} first defines the closure operator $\cl$, and then defines a flat as a subset $X$ of $M$ for which $\cl(X) = X$ (that is, $X$ is a closed subset, see \ref{def:closedSubsets}).
This definition is in fact invoked in the context of independence spaces, which are pi-spaces $(M,I(M))$ that additionally satisfy the following property:

\begin{enumerate}
    \item[(FC)] If $\sigma\subseteq M$ and every finite subset of $\sigma$ lies in $I(M)$, then $\sigma\in I(M)$.
\end{enumerate}

These structures are also called \textit{finitary matroids} (see, for instance, \cite{Bean, Higgs, Klee, OxleyOld, Oxley}).

We prove that for finite-rank subsets, both definitions of flat coincide in the context of pi-spaces (see \ref{coro:finiteRankFlats} and \ref{thm:homotopyTypeClosedSets}(1)).
Also, for finitary matroids, flats and closed subsets are the same.

In any case, if $\C(M,I(M))$ denotes the poset of closed subsets of $M$, we recover an analogue of \ref{thm:homotopyPosetOfFlats} with $\L(M,I(M))$ replaced by $\C(M,I(M))$.
See \ref{thm:homotopyTypeClosedSets}.

Second, ``infinite matroids" have also been considered in other works \cite{infM, OxleyOld, Oxley, Welsh}.
For instance, the definition of an infinite matroid given in \cite{infM} requires condition (M) there, which assures the existence of certain maximal elements in $I(M)$.
In this paper, we will not discuss the possible definitions of infinite matroids, but just focus on the homotopy type of the associated posets for pi-spaces and finitary matroids.

Finally, we mention that in the finite-dimensional case of \ref{thm:homotopyIndependentSetsPoset}, the poset $I(M)$ may even be shellable. A notion of shellability for non-compact but finite-dimensional simplicial complexes is discussed, for instance, in Remark 4.21 of \cite{BjornerInf}.

For the rest of the paper, whenever we speak of a matroid, we mean a matroid in the classical finite sense with the exception of the term ``finitary matroid", which is an established notion from the literature.

The paper is organized as follows.
In \ref{sec:pi-spaces}, we define pre-independence spaces and finitary matroids, and derive properties of $I(M)$ that culminate in the proof of \ref{thm:homotopyIndependentSetsPoset}.
In \ref{sec:flats}, we provide our notion of flats, circuits, and closure. We derive properties of the poset of flats, and its relation with the closure operator (in the general context of pi-spaces, a closed set is not always a flat). We include intermediate results which
smoothen the way to the proof of \ref{thm:homotopyPosetOfFlats}.
In \ref{sec:closedSets}, we discuss further results on the poset of closed subsets together with some examples.
Finally, we include an appendix, \ref{app:homotopy}, with notation and
definitions for posets and simplicial complexes, and homotopy tools.

\bigskip

\noindent
\textbf{Acknowledgments.}
The first author was supported by the FWO grant 12K1223N.

\section{Pre-independence spaces} \label{sec:pi-spaces}

In this section, we recall some of the main definitions of \cite[Chapter 20]{Welsh} and set up the notation.
We also provide the proof of \ref{thm:homotopyIndependentSetsPoset}.
We refer the reader to the appendix for general notation on posets and simplicial complexes, as well as for the topological notions invoked here.

\begin{definition}
A \textit{pre-independence space (pi-space)} is a pair $(M,I(M))$ where $M$ is a non-empty set and $I(M)$ is a non-empty set of subsets of $M$ satisfying the following two conditions:
\begin{enumerate}
    \item[(I1)] If $\sigma \in I(M)$ then every subset of $\sigma$ lies in $I(M)$.
    \item[(I2)] If $\sigma,\tau \in I(M)$ are finite and $|\sigma|>|\tau|$ then there exists $x\in \sigma \setminus \tau$ such that $\tau \cup \{x\}\in I(M)$.
\end{enumerate}
Condition (I2) is usually called the 
``Exchange property".

The elements of $I(M)$ are called \textit{independent sets}, and we regard $I(M)$ as a poset with ordering given by inclusion.
A subset of $M$ is called \textit{dependent} if it is not in $I(M)$.
\end{definition}

If $I(M)$ contains no infinite sets, then $I(M)$ is a simplicial complex and we shall call $(M,I(M))$ a \textit{simplicial pi-space}.
In general, we consider $I(M)$ as a poset ordered by inclusion. 
First, note that $\emptyset \in I(M)$ is the unique minimal element of $I(M)$.
By a slight abuse of notation, we shall denote by $|I(M)|$ the order complex of the poset $I(M)\setminus \{\emptyset\}$.
When we speak of the dimension and topology of $I(M)$ we will mean
the dimension and topology of $|I(M)|$.
The \textit{rank} of $(M,I(M))$, denoted by $\rk_{I(M)}(M)$, or simply $\rk(M)$ when $I(M)$ is implicit, is the dimension of $I(M)$ plus one.
Observe that when $\rk(M)$ is finite, $I(M)$ is a simplicial complex, $|I(M)|$ is the barycentric subdivision of $I(M)$, and $\rk(M)$ is one plus the dimension of $I(M)$.

\begin{remark}
\label{rk:alternativeExchange}
If $(M,I(M))$ is a pi-space of finite rank, then condition (I2) is equivalent to:
\begin{enumerate}
    \item[(I2')] If $\sigma\in I(M)$ is a maximal simplex and $\tau \in I(M)$ is such that $|\sigma|>|\tau|$ then there exists $x\in \sigma\setminus \tau$ such that $\tau \cup \{x\}\in I(M)$.
\end{enumerate}
\end{remark}


Below, we recall the definition of finitary matroids
(see \cite{Bean, Higgs, Klee, Oxley, Welsh}).

\begin{definition}
\label{def:finitaryMatroid}
Let $(M,I(M))$ be a pi-space.
We say that $(M,I(M))$ is a \textit{finitary matroid}, or an \textit{independence space}, if the following condition holds:
\begin{enumerate}
    \item[(FC)] If $\sigma \subseteq M$ and all finite
    subsets of $\sigma$ are in $I(M)$, then $\sigma \in I(M)$.
\end{enumerate}
\end{definition}

Note that a finite-rank pi-space is always a finitary matroid.

\begin{proposition}[Theorem 3, p. 387 \cite{Welsh}]
\label{prop:finitaryMatroids}
Let $(M,I(M))$ be a finitary matroid.
Then $(M,I(M))$ has bases (i.e., maximal independent sets), two of them have the same cardinality, and every independent set is contained in a basis.
We write $\B(M,I(M))$ for the set of bases of $(M,I(M))$.
\end{proposition}

\begin{remark}
[Infinite matroids]
There are several approaches to infinite matroids (see \cite{Oxley, Welsh}).
A more recent approach to the study of the combinatorics of such structures was given in \cite{infM}.
The Independence Axioms stated in Section 1.1 of \cite{infM} define a (possibly infinite) matroid to be a pair $(M,I(M))$ with the following properties: $I(M)$ is non-empty and downward closed (i.e., (I1) holds);
for every non-maximal independent set $\tau \in I(M)$ and maximal independent set $\sigma\in I(M)$, we can find $x\in \sigma\setminus \tau$ such that $\tau\cup \{x\}\in I(M)$; and
certain extra maximality condition holds, denoted by (M) there.
In the context of \cite{infM}, infinite matroids are not necessarily finitary (so (FC) may fail).
The absence of such a property and the presence of condition (M) somehow make duality possible for infinite matroids, settling a problem raised by Rado (see \cite[Problem P531]{Rado}).

It is not hard to verify that, when $(M,I(M))$ is a finite-rank pi-space, $(M,I(M))$ is a matroid in the sense of \cite{infM}.
In particular, one can invoke results from \cite{infM} in such a case.
\end{remark}

We continue with some properties for pi-spaces that will be useful for the proof of \ref{thm:homotopyIndependentSetsPoset}.
The following proposition includes some of the properties that we have already mentioned and describes some new pi-spaces defined from a given pi-space.

\begin{proposition}
\label{prop:basics}
Let $(M,I(M))$ be a pi-space.
Then the following hold:
\begin{enumerate}
    \item If every independent set is finite, then $I(M)$ is a simplicial complex.
    \item If $\rk(M)$ is finite, then $I(M)$ is a finite-dimensional and pure simplicial complex.
    \item If $N\subseteq M$ is a non-empty subset and $I(M,N)$ consists of the independent sets in $I(M)$ whose elements lie in $N$, then $(N,I(M,N))$ is a pi-space.
    If $(M,I(M))$ is a finitary matroid, then so is $(N, I(M,N))$.
    \item For any $k\geq 0$, the $k$-skeleton of $I(M)$, namely the set of independent sets of size $\leq k+1$, is a finite-rank pi-space on $M$.
    We shall denote this by $(M,I(M)^{(k)})$.
    \item If $\sigma \in I(M)$ then $$\Lk_{I(M)}(\sigma) = \{\tau \in I(M)\tq \sigma\cup \tau \in I(M), \, \tau\cap\sigma = \emptyset\}$$ is a pi-space on $M \setminus \sigma$.
    If $\sigma$ is finite then $\dim \Lk_{I(M)}(\sigma) = \dim I(M) - |\sigma|$.
\end{enumerate}
\end{proposition}

\begin{proof}
These facts are straightforward to verify.
\end{proof}

The following two structures will also be useful.

\begin{definition}
Let $(M,I(M))$ be a pi-space.
We define:
\begin{itemize}
    \item $I(M)_{\fin} = \{\sigma\in I(M)\tq |\sigma|<\infty\}$.
    \item $I(M)_W = \{ \sigma\subseteq M \tq $ every finite subset of $\sigma$ lies in $I(M)\}$.
\end{itemize}

We call $(M,I(M)_{\fin})$ the \textit{simplicial pi-space associated with $(M,I(M))$}.
\end{definition}

The proof of the next lemma follows easily from the definitions above.

\begin{lemma}
Let $(M,I(M))$ be a pi-space.
Then the following hold:
\begin{enumerate}
    \item $I(M)_{\fin} \subseteq I(M) \subseteq I(M)_W$.
    \item $(M,I(M)_{\fin})$ is a simplicial pi-space, and  $(M,I(M))$ is simplicial if and only if $I(M) = I(M)_{\fin}$.
    \item $(M,I(M)_W)$ is a finitary matroid, and $(M,I(M))$ is a finitary matroid if and only if $I(M) = I(M)_W$.
    \item If $\rk_{I(M)}(M) < \infty$, then $I(M)_{\fin} = I(M) = I(M)_W$.
\end{enumerate}
\end{lemma}

Before we continue with the analysis of the homotopy type of $I(M)$, let us take a look at some familiar examples.

\begin{example}
\label{ex:vectorspace}
Let $V$ be a non-zero vector space, and let $I(V)$ consist of the linearly independent subsets of $V$.
Then $(V,I(V))$ is a finitary matroid.

On the other hand, $(V,I(V)_{\fin})$ is a finitary matroid if and only if $V$ is finite-dimensional.
\end{example}

\begin{example}
\label{ex:graphicmatroid}
Let $G$ be a $1$-dimensional CW-complex (i.e., an undirected graph possibly with multiple edges and loops).
For simplicity, we assume that $G$ has no isolated vertices.
Let $E(G)$ be the set of $1$-cells of $G$ (i.e., the edges).
Then we let $I(G)$ be the set of subcomplexes of $G$ that are forests (i.e., disjoint unions of (not necessarily compact) contractible 
subcomplexes).

It is not hard to see that $(E(G), I(G))$ is a pi-space---indeed, for (I2), one can argue exactly as in the finite case.
Also note that (FC) holds: if $\sigma\subseteq E(G)$ and $H$ denotes the subcomplex spanned by $\sigma$, then, by a compactness argument, $H$ is a forest if and only if every finite subcomplex of $H$ is a forest.
Thus, $(E(G),I(G))$ is a finitary matroid.
\end{example}

The following example analyzes the homotopy type of $I(M)$ in a particular case.
Recall that, to study the homotopy type of posets, we remove the unique minimal element (if any), which is given by the empty set for the posets $I(M)$.

\begin{example}
\label{ex:allSubsets}
Let $M$ be an infinite set, $I_1(M)$ be the set of finite subsets of $M$, and $I_2(M)$ be the set of all proper subsets of $M$.
Then $I_1(M)_{W} = I_2(M)_W =  I_2(M) \cup \{M\} = 2^M$ is the poset of all subsets of $M$, and $I_2(M)_{\fin}=I_1(M)$.
Note that $I_1(M)$ is a contractible poset:  we can pick an element $x\in M$, and the chain of inclusions $\sigma \subseteq \sigma \cup \{x\} \supseteq \{x\}$ gives rise to a homotopy equivalence between the identity map and the constant map.

For $I_2(M)$, note that the inclusion $i:I_1(M)\hookrightarrow I_2(M)$ is a homotopy equivalence by Quillen's fiber \ref{thm:quillen}: if $\sigma\in I_2(M)$ is infinite, then $i^{-1}(I_2(M)_{\subseteq \sigma}) = I_1(\sigma)$, which is contractible.
Thus, $I_2(M)$ is also contractible.
\end{example}

We show now that the proof given in the previous example generalizes, relating the homotopy type of $I(M)_{\fin}$ to that of $I(M)$ and $I(M)_W$.

\begin{proposition}
\label{prop:finReduction}
Let $(M,I(M))$ be a pi-space.
Then the inclusions $I(M)_{\fin} \hookrightarrow I(M)$ and $I(M) \hookrightarrow I(M)_W$ are homotopy equivalences.
\end{proposition}

\begin{proof}
Let $i: I(M)_{\fin} \hookrightarrow I(M)$ and $j: I(M)_{\fin} \hookrightarrow I(M)_W$ denote the inclusions.
We use Quillen's fiber theorem and show that $i,j$ are homotopy equivalences.
From this, it will follow that $I(M)\hookrightarrow I(M)_W$ is also a homotopy equivalence.

Let $\sigma \in I(M)_W\setminus I(M)_{\fin}$, so $\sigma$ is infinite.
Then $j^{-1}({I(M)_W}_{\subseteq \sigma})=i^{-1}({I(M)}_{\subseteq \sigma})$ is the set of all finite subsets of $\sigma$.
By \ref{ex:allSubsets}, this is a contractible poset (recall that, to compute its homotopy type, we remove the empty set).
If $\sigma \in I(M)_{\fin}$, then this preimage contains $\sigma$ as a unique maximal element, and hence it is also contractible.
Thus, by Quillen's fiber theorem, $i$ and $j$ are homotopy equivalences.
\end{proof}

Now we prove \ref{thm:homotopyIndependentSetsPoset}.

\begin{proof}[Proof of \ref{thm:homotopyIndependentSetsPoset}]
By \ref{prop:finReduction}, we see that $I(M)\simeq I(M)_{\fin}$.
Therefore, it is enough to prove the theorem for $(M,I(M))$ a simplicial pi-space.
Also, since the $k$-skeleton of $I(M)$ is a finite-rank pi-space by \ref{prop:basics}, we see that item (1) implies item (2) of the theorem.

Thus, we assume that $I(M)$ is finite-dimensional and prove item (1).
Let $n = \dim I(M)$.
We show that the homotopy group $\pi_k(|I(M|),x)$ is trivial for all $k<n$, $x\in |I(M)|$.
Let $f:S^k \to |I(M)|$ be a continuous map.
By compactness, the set
\begin{center}
$\{\sigma \in I(M) \tq $ exists $z\in S^k$ such that $f(z)$ lies in the interior of the cell $\sigma\}$
\end{center}
is finite.
Hence, there is a finite subset $N\subseteq M$ such that $f$ factorizes through $|I(M,N)|$:
\[ 
\xymatrix{
& S^k \ar@{..>}[ld]_{\exists f'} \ar[d]^f \\
|I(M,N)| \ar@{^(->}[r] & |I(M)|
}
\]
By \ref{prop:basics}, $I(M,N)$ is a matroid on the set $N$, so it is Cohen-Macaulay of dimension $k' := \dim I(M,N)$.
This means that its homotopy groups of degree $< k'$ vanish.
Therefore, $f':S^k\to |I(M,N)|$ (and also $f$) is null-homotopic provided that $k' > k$.

Now we show that we can always take $N$ such that $k' > k$.
Indeed, if $I(M,N)$ happens to have dimension $k' \leq k$, then, since $I(M,N)$ is finite and $k<n$, we can take for every maximal simplex $\sigma$ of $I(M,N)$ a simplex $\sigma'\in I(M)$ such that $|\sigma'|=|\sigma|+1$.
Then we can consider the finite subset $N'$ obtained by adding all the vertices of the $\sigma'$s to $N$.
Hence, $I(M,N')$ is (at least) $(k'+1)$-dimensional, and contains $I(M,N)$.
Continuing in this way, we can produce a finite set $N'$ such that $I(M,N) \subseteq I(M,N')$ and the latter has dimension $> k$.

The above argument establishes the sphericity of $I(M)$.
By \ref{prop:basics}(4), the link of any simplex $\sigma\in I(M)$ is also spherical of dimension $\dim I(M) - |\sigma|$.
This shows that $I(M)$ is Cohen-Macaulay.
\end{proof}

The previous proof is based on the fact that matroids are Cohen-Macaulay, a result obtained by providing a shelling of the independence complex.
One can define shelling for non-compact finite-dimensional simplicial complexes, as it is explained, for instance, in Remark 4.21 of \cite{BjornerInf}.
Precisely, if $K$ is a finite-dimensional and pure simplicial complex, a shelling on $K$ is a well-ordering $I$ of its maximal simplices, say $(\sigma_i)_{i\in I}$, such that if $K_i$ is the subcomplex with maximal faces $\sigma_j$ for $j<i$, then $\sigma_i \cap K_i$ is a pure complex of dimension $|\sigma_i|-2$.
A pure and finite simplicial complex is the independence complex of a matroid if and only if any total ordering of its vertices induces a shelling (see \cite{Bjorner1}).

On the other hand, although shellability implies Cohen-Macaulayness, the latter homotopical condition can also be deduced easily from the combinatorial definition of a matroid without going through a shelling.
Indeed, assume that $(M,I(M))$ is a matroid of rank $n$ over a finite set $M$, and suppose that we have inductively shown that $I(N)$ is spherical for any matroid $(N,I(N))$ with $|N|<|M|$.
Pick any element $x\in M$, and consider:
\[ I(M\setminus x) := I(M,M\setminus \{x\}) = \{\sigma\in I(M) \tq x\notin \sigma\},\]
and
\[ \Lk_{I(M)}(x) = \{\sigma \in I(M\setminus x) \tq  \sigma\cup\{x\} \in I(M)\}.\]
Assume for a moment that $\{x\} \in I(M)$, and that $\{x\}$ is not the unique non-empty independent set.
Then $I(M) = I(M\setminus x) \cup (x * \Lk_{I(M)}(x))$.
Since $I(M\setminus x)$ and $\Lk_{I(M)}(x)$ are easily seen to be matroids of rank $n$ and $n-1$ respectively (cf. \ref{prop:basics}), they are both spherical.
The intersection $I(M\setminus x) \cap (x*\Lk_{I(M)}(x)) = \Lk_{I(M)}(x)$ is null-homotopic in $I(M\setminus x)$, and since $x*\Lk_{I(M)}(x)$ is contractible, we conclude that $I(M)$ is homotopy equivalent to the wedge of $I(M\setminus x)$ and the suspension of $\Lk_{I(M)}(x)$, and hence it is spherical.

In the case $\{x\}\notin I(M)$, then no simplex of $I(M)$ contains $x$, so $(M\setminus \{x\},I(M))$ is a matroid with fewer elements, and by induction $I(M)$ is spherical.

As a consequence, from \ref{ex:vectorspace} and \ref{ex:graphicmatroid} we get:

\begin{corollary}
\label{coro:vectorspaces}
Let $V$ be a vector space, and let $I(V)$ be the complex of linearly independent subsets of $V$.
Then $I(V)$ is Cohen-Macaulay if $V$ is finite-dimensional, and it is contractible otherwise.
\end{corollary}

\begin{corollary}
\label{coro:graphicmatroid}
Let $G$ be a $1$-dimensional CW-complex without isolated vertices, and let $I(G)$ be the poset consisting of subcomplexes of $G$ that are forests.
Then $I(G)$ is Cohen-Macaulay if $G$ has a finite number of vertices, and it is contractible otherwise.
\end{corollary}


We also get results on the homotopy type of ``injective word complexes".
Recall that if $K$ is a simplicial complex, the complex of injective words on $K$ is the poset $\Ord K$ whose elements are tuples $(v_1,\ldots,v_r)$ of distinct vertices such that $\{v_1,\ldots,v_r\}$ is a simplex of $K$.
For two elements of this poset, say $x,y\in \Ord K$, we have that $x\leq y$ if the tuple $x$ is obtained by deleting some elements of the tuple $y$.

More generally, let $M$ be a set and $I(M)$ a family of subsets of $M$ closed under inclusion.
We can define $\U(I(M))$ to be the poset whose underlying set is
\[ \{ (X,\preceq) \tq X\in I(M) \setminus \{\emptyset\} \text{ and $\preceq$ is a total ordering on $X$}\}.\]
The ordering in $\U(I(M))$ is given by $(X,\preceq_X)\leq (Y,\preceq_Y)$ if $X\subseteq Y$ and $\preceq_X$ is just $\preceq_Y$ when restricted to $X$.
If the elements of $I(M)$ are finite, then $I(M)$ is a simplicial complex and $\U(I(M)) = \Ord I(M)$ is the complex of injective words as described above.

We are interested in describing the homotopy type of $\U(I(M))$ when $(M,I(M))$ is a pi-space.
This will follow from our results on the homotopy type of $I(M)$.

\begin{corollary}
\label{coro:injectiveWordsHomotopy}
Let $(M,I(M))$ be a pi-space.
\begin{enumerate}
    \item If $M$ has finite rank, then $\U(I(M))$ is Cohen-Macaulay.
    \item If $M$ has infinite rank, then $\U(I(M))$ is contractible.
\end{enumerate}
Moreover, the following inclusions are homotopy equivalences:
\[ \U(I(M)_{\fin}) \hookrightarrow \U(I(M)) \hookrightarrow \U(I(M)_W).\]
\end{corollary}

\begin{proof}
If $M$ has finite rank, then $I(M)$ is a simplicial complex, and $\U(I(M))$ is the complex of injective words on $I(M)$.
Thus, item (1) follows from Theorem 5.9(3) of \cite{PW} and \ref{thm:homotopyIndependentSetsPoset}(1).
In this case, $I(M)_{\fin} = I(M) = I(M)_W$, which establishes the Moreover part under the assumption that $M$ has finite rank.

Next, suppose that $M$ has infinite rank.
We show first that $X := \U(I(M)_{\fin})$ is contractible.
Let $X^{(k)}$ denote the subposet of $X$ consisting of tuples $x\in X$ with at most $k+1$ elements, that is, $X^{(k)} = \U(I(M)^{(k)})$.
By item (1), $X^{(k)}$ is Cohen-Macaulay of dimension $k$, and in particular it is $(k-1)$-connected.

Now, let $f:S^n\to |\U(I(M)_{\fin})|$ be a continuous map from the $n$-dimensional sphere $S^n$ to the geometric realization of $\U(I(M)_{\fin})$.
By compactness, the image of $f$ lies in some $|X^{(k)}|$, and we can even take $k > n$.
Then we get a continuous map $\tilde{f}:S^n\to |X^{(k)}|$ such that $f$ is the composition of $\tilde{f}$ with the inclusion $|X^{(k)}|\hookrightarrow |\U(I(M)_{\fin})|$.
By connectivity of $X^{(k)}$, $\tilde{f}$ is null-homotopic, and therefore $f$ is null-homotopic.
Since we are working with CW-complexes, this proves that the $n$-th homotopy group of $\U(I(M)_{\fin})$ is the trivial group (regardless of the basepoint).
Therefore, $\U(I(M)_{\fin})$ is contractible.

Next, we prove that if $\sigma \in \U(I(M)_W)\setminus \U(I(M)_{\fin})$, then the preimage of the lower interval on $\sigma$ by the inclusion $i: \U(I(M)_{\fin}) \hookrightarrow \U(I(M)_W)$ is contractible.
Note that
\[ i^{-1}(\U(I(M)_W)_{\leq \sigma}) = \U(I(M,\sigma)_{\fin}).\]
Since $\sigma$ is infinite, $\U(I(M,\sigma)_{\fin})$ is contractible by the previous paragraph.
Hence, by Quillen's fiber theorem, $i$ is a homotopy equivalence.

The same argument shows that the inclusion $\U(I(M)_{\fin}) \hookrightarrow \U(I(M))$ is a homotopy equivalence.
This proves the Moreover part when $M$ has infinite rank.
In fact, this shows that $\U(I(M)_{\fin})$, $\U(I(M))$ and $\U(I(M)_{W})$ are contractible when $M$ has infinite rank, so item (2) holds.
\end{proof}

This corollary recovers part of the results from \cite[2.6]{vdk} in the case of fields.
Moreover, if $V$ is an infinite-dimensional vector space and $I(V)$ is the poset of linearly independent subsets of $V$, the previous corollary shows that $\U(I(V)_{\fin})$ and $\U(I(V))$ are contractible.

\section{Flats, closure and circuits}
\label{sec:flats}

The purpose of this section is to define the poset of flats of a pi-space, the notion of closure, and circuits.
We will establish some properties relating these notions that are inspired by the finite case in matroid theory.
Then we will use them to provide the proof of \ref{thm:homotopyPosetOfFlats}.
Throughout this section, $(M,I(M))$ denotes a pi-space.

\begin{definition}
Let $(M,I(M))$ be a pi-space.
The \textit{rank} of a subset $F\subseteq M$ is the supremum over the sizes of the independent sets of $M$ contained in $F$.
We denote this number by $\rk_{I(M)}(F)$ (or simply by $\rk(F)$ if $I(M)$ is implicit), which can take the value $\infty$.
We say that $F$ is a \textit{flat} if for any $x\in M\setminus F$ and any independent set $\sigma\subseteq F$ we have $\sigma \cup \{x\} \in I(M)$.

We let $\L(M, I(M))$ denote the poset of flats of $M$, with order given by inclusion.
\end{definition}

The elements of $M$ which do not lie in any independent set are called \textit{loops}, and they constitute the set $0_{I(M)} = \{x\in M \tq \{x\}\notin I(M)\}$, which can be empty.
Clearly, $0_{I(M)}$ is a flat.
Since $0_{I(M)} \subseteq F$ for any flat $F$, $0_{I(M)}$ is the unique minimal element of $\L(M,I(M))$.

On the other hand, we obviously have $M\in \L(M,I(M))$, and $M$ is the unique maximal element of $\L(M,I(M))$.
Thus we get:

\begin{lemma}
\label{lm:boundedFlats}
For a pi-space $(M,I(M))$, $\L(M, I(M))$ is a bounded poset with maximal element $M$, and unique minimal element $0_{I(M)}$, the set of all loops of $M$.
\end{lemma}

Before we move on to study further properties of $\L(M,I(M))$, let us take a look at one example.

\begin{example}
\label{ex:vectorSpacesAndFlats}
Let $V$ be a vector space, and let $I(V)$ consist of the linearly independent subsets of $V$ (see \ref{ex:vectorspace}).
Then, $\L(V,I(V))$ is exactly the poset of subspaces of $V$.
Since elements of $V$ are finite linear combinations of linearly independent vectors, every subspace of $V$ is an element of $\L(V,I(V)_{\fin})$.
Hence $\L(V,I(V)_{\fin}) = \L(V,I(V))$. 
\end{example}

In view of \ref{lm:boundedFlats}, we want to analyze the homotopy type of the proper part
\[ \red{\L(M, I(M))} = \L(M, I(M)) \setminus \{M,0_{I(M)}\}. \]

The following is immediate:

\begin{lemma}
\label{lm:lower_interval_flats}
Let $(M,I(M))$ be a pi-space, and let $F$ be a flat of $M$.
Then $(F,I(M,F))$ is a pi-space and $\L(F, I(M,F)) = \L(M, I(M))_{\subseteq F}$.
\end{lemma}

For the upper intervals, we need to introduce the concept of contraction, which works nicely in the context of finitary matroids (cf. \ref{prop:finitaryMatroids}).

\begin{definition}
Let $(M,I(M))$ be a finitary matroid, and let $X\subseteq M$.
The \textit{contraction of $M$ by $M \setminus X$} is the pair $(X, I(M.X))$, where $I(M.X)$ consists of the independent sets $\sigma\in I(M,X)$ such that there exists a maximal independent set $\tau$ of the finitary matroid $(M\setminus X, I(M,M\setminus X))$ for which $\sigma\cup \tau\in I(M)$.
\end{definition}

Indeed, the contraction is a finitary matroid again.

\begin{proposition}
\label{prop:contractionIsFinitary}
Let $(M,I(M))$ be a finitary matroid, and $X\subseteq M$.
Then the following hold:
\begin{enumerate}
    \item $(X,I(M.X))$ is a finitary matroid.
    \item If $\tau \in I(M,M\setminus X)$ and $\sigma \in I(M.X)$, then $\sigma\cup\tau\in I(M)$.
\end{enumerate}
\end{proposition}

\begin{proof}
These items follow respectively from Proposition 3.1.9 and Lemma 3.1.7 of \cite{Oxley}.
\end{proof}

The following lemma describes the upper intervals of finitary matroids in terms of contractions, as it usually holds in the context of matroids.

\begin{lemma}
\label{lm:upper_interval_flats}
Let $(M,I(M))$ be a finitary matroid, and let $F\subseteq M$ be a flat.
Then the map
\[ G \in \L(M,I(M))_{\supseteq F} \longmapsto G\setminus F\in \L(M\setminus F,I(M.(M\setminus F)))\]
is an isomorphism of posets.

In particular, if $\rk_{I(M)}(M) < \infty$, $(M\setminus F,I(M.(M\setminus F)))$ has rank $\rk_{I(M)}(M) - \rk_{I(M)}(F)$.
\end{lemma}

\begin{proof}
Write $X = M\setminus F$.
For $G\in \L(M,I(M))_{\supseteq F}$, let $\psi(G) := G\setminus F$.
It is straightforward to verify that $\psi(G) \in \L(X, I(M.X))$, so $\psi$ is an order-preserving and injective map.
Also, $\psi(G_1) \subseteq \psi(G_2)$ implies that $G_1\subseteq G_2$.
Thus, to show that $\psi$ is an isomorphism, it remains to see that $\psi$ is surjective.
In fact, the inverse of $\psi$ should be given by $S\mapsto S\cup F$, as we will see next.

Let $S\in \L(X, I(M.X))$, and set $G := S\cup F$.
We show that $G$ is a flat for $(M,I(M))$ (which clearly contains $F$).
Let $x\in M\setminus G$, and let $\sigma\in I(M,G)$ be an independent set.
We must prove that $\sigma\cup \{x\}\in I(M)$.
As $(M,I(M))$ is a finitary matroid, by \ref{prop:finitaryMatroids} it is enough to show this for $\sigma \in \B(G,I(M,G))$ a basis of $G$.

Let $\tau \in \B(F,I(M,F))$, and set $E := \sigma\cup F$, $I(E) := I(M,E)$, $E' := E\cup \{x\}$ and $I(E') := I(M,E')$.
Note that $\sigma \in \B(E,I(E))$ since $E\subseteq G$ and $\sigma$ is already a basis for $G$.
Then, as $\tau\subseteq E$, by \ref{prop:finitaryMatroids} there exists a subset $\sigma'$ of $\sigma\setminus F$ such that $\sigma'\cup\tau$ is a basis of $(E, I(E))$.
By definition, $\sigma'$ is an independent set of $(X,I(M.X))$ contained in the flat $S$.
As $x\in X\setminus S$, we see that $\sigma'\cup \{x\}\in I(M.X)$.
By \ref{prop:contractionIsFinitary}, $\sigma'\cup \tau\cup \{x\}$ is an independent set of $M$ contained in $E'$.
Since $\sigma'\cup \tau \in \B(E, I(E))$, $\sigma'\cup \tau\cup \{x\}$ must be a basis of $(E',I(E'))$.

Now we argue by the way of contradiction and suppose that $\sigma\cup \{x\}$ is a dependent set.
This implies that $\sigma$ is a basis of $(E', I(E'))$.
Let $B_1 := \sigma'\cup \tau \cup \{x\}$ and $B_2 := \sigma$.
Then $B_1,B_2 \in \B(E', I(E')) $, with $x\in B_1\setminus B_2$.
By Lemma 3.1.6 of \cite{Oxley}, there is $y\in B_2\setminus B_1 = \sigma\setminus (\sigma'\cup \tau)$ such that $(B_1\setminus \{x\})\cup \{y\}$ is an independent set (even a basis) of $(E',I(E'))$.
But then $(\sigma'\cup \{y\})\cup \tau$ is an independent set of $(E,I(E))$ strictly containing the basis $\sigma'\cup \tau$, a contradiction.
Therefore, $\sigma\cup \{x\}$ is an independent set.

This proves that $\psi$ is surjective, and hence an isomorphism of posets.
The In particular part follows by just looking at the previous proof in the finite-rank case.
This concludes the proof of this lemma.
\end{proof}

We are not aware of a similar description of upper intervals for arbitrary pi-spaces.

We now turn our attention to combinatorial properties of the poset of flats $\L(M,I(M))$.

\begin{lemma}
\label{lm:intersectionFlats}
Let $(F_i)_{i\in I}$ be a collection of flats of a pi-space $(M,I(M))$.
Then $\bigcap_{i\in I} F_i$ is a flat.

In particular, $\L(M,I(M))$ is a complete lattice.
\end{lemma}

\begin{proof}
Denote $F = \bigcap_i F_i$, and take $x\in M \setminus F$ and $\sigma\subseteq F$ an independent set.
Then there exists $i$ such that $x\notin F_i$.
Since $\sigma\subseteq F_i$, we have $\sigma\cup \{x\}\in I(M)$.
Thus $F$ is a flat.

Let $X\subseteq \L(M,I(M))$.
Then $\bigvee_{F\in X} F$ is the intersection of all upper bounds of $X$,
and $\bigwedge_{F\in X} F = \bigcap_{F\in X} F$.
\end{proof}

When $(M,I(M))$ is a finite-rank pi-space, we recover the geometric lattice (i.e., atomistic and semimodular) properties of the lattice 
of flats of a matroid, as stated below.

\begin{proposition} \label{prop:geometric}
Let $(M,I(M))$ be a finite-rank pi-space.
Then $\L(M,I(M))$ is a ranked, atomistic and semimodular lattice.
That is:
\begin{enumerate}
    \item $\L(M,I(M))$ is a lattice with rank function $\rk_{I(M)}$.
    \item $\L(M,I(M))$ is atomistic: every flat is the join of the atoms below it.
    \item $\L(M,I(M))$ is semimodular: for all $F,G\in \L(M,I(M))$ we have $\rk_{I(M)}(F)+\rk_{I(M)}(G) \geq \rk_{I(M)}(F\vee G) + \rk_{I(M)}(F\wedge G)$. 
\end{enumerate}
\end{proposition}

The proof of \ref{prop:geometric} is identical with 
the proof of Theorem 5 in \cite[Chapter 20]{Welsh}. It will require
though that under its assumptions, being a flat and being closed are
equivalent. This fact is only proved later in \ref{thm:homotopyTypeClosedSets}.
Since \ref{prop:geometric} is not used anywhere in the text and fits perfectly in its spot, we 
consider this deviation from standard style acceptable.

Next, we look at circuits and the closure operator.

\begin{definition}
Let $(M,I(M))$ be a pi-space.
\begin{enumerate}
    \item A \textit{circuit} is a subset $C\subseteq M$ that is a minimal dependent set.
    \item The \textit{closure} of a subset $X\subseteq M$ is $$\cl(X) = X \cup \{x\in M \tq \text{there exists a circuit $C$ such that } x\in C\subseteq X\cup \{x\} \}.$$
\end{enumerate}
We may write $\cl_{I(M)}$ to emphasize that the closure is taken in the pi-space $(M,I(M))$.
\end{definition}

\begin{remark}
\label{rk:closureAndloops}
Circuits of size one are exactly the sets $\{x\}$ where $x$ is a loop of $(M,I(M))$.
Therefore, the closure of the empty independent set is $\cl(\emptyset) = 0_{I(M)}$, the set of all loops.
More generally, $0_{I(M)} \subseteq \cl(X)$ for any subset $X\subseteq M$.
If $X$ consists only of loops, then we also have $\cl(X) = 0_{I(M)}$.
\end{remark}

Note that we may have an infinite descending chain of dependent sets that contain no circuits.

\begin{example}
\label{ex:finiteSubsetsIndependent}
Suppose $I(M)$ is the set of finite subsets of $M$, where $M$ is an infinite set.
Therefore, $(M,I(M))$ is a pi-space.

Note that a subset $X\subseteq M$ is dependent if and only if $X$ is infinite.
As every infinite set has an infinite proper subset, we conclude that there are no circuits in this pi-space.
Moreover, we can produce infinite descending chains of dependent sets.
\end{example}

More generally, we have:

\begin{example}
Suppose that $M$ is an infinite set of cardinality $\alpha$.
Let $I(M)$ consist of the subsets of $M$ of cardinality $<\alpha$.
Then $(M,I(M))$ is a pi-space where dependent sets are subsets of $M$ of cardinality $\neq\alpha$.
As $M$ is infinite, every subset of cardinality $\alpha$ contains a proper subset of cardinality $\alpha$.
Thus, $(M,I(M))$ contains no circuits.
Note that this is not a finitary matroid.
\end{example}

\begin{example}
Let $V$ be a non-zero vector space, and let $(V,I(V))$ be the pi-space defined in \ref{ex:vectorspace} (see also \ref{ex:vectorSpacesAndFlats}).
Then a circuit is exactly a minimal linearly dependent set of non-zero vectors of $V$.
In particular, circuits are finite.
\end{example}

Expanding on the previous example, finitary matroids always have finite circuits, as stated in the following lemma, whose proof is straightforward.

\begin{lemma}
\label{lm:finitaryAndCircuits}
If $(M,I(M))$ is a finitary matroid, then circuits are finite, and every dependent set contains a circuit.
\end{lemma}


Note that if $I(M)$ consists of all subsets of $M$, then $(M,I(M))$ is a finitary matroid but contains no circuits.

We have seen examples of pi-spaces containing dependent sets but no circuits.
The next example shows that circuits might be infinite if $(M,I(M))$ is not a finitary matroid.

\begin{example}
Suppose $M$ is an infinite set, and take $I(M)$ to be the set of all proper subsets of $M$.
Thus $(M,I(M))$ is a pi-space with exactly one circuit, namely $M$, which is infinite.
Clearly, $(M,I(M))$ is not a finitary matroid.
\end{example}

Recall that if $(M,I(M))$ is a matroid, then $F\subseteq M$ is a flat if and only if $F = \cl(F)$.
Moreover, for any subset $X\subseteq M$, $\cl(X)$ is a flat, so $\cl(\cl(X)) = \cl(X)$ is an order-preserving idempotent operator.
In the general context of pi-spaces, $\cl$ is still an order-preserving operator, but it might not be idempotent in general, as we will see later in \ref{ex:infiniteAscChainNonClosed}.

The next lemma shows that $\cl$ is the identity on flats, and so $\cl(\cl(X)) = \cl(X)$ if $\cl(X)$ is a flat.

\begin{lemma}
\label{lm:flatClosure}
Let $(M,I(M))$ be a pi-space, and $F\subseteq M$.
If $F$ is a flat, then $F = \cl(F)$.
The converse holds if $(M,I(M))$ is a finitary matroid.
\end{lemma}

\begin{proof}
We prove that if we have $x\in \cl(F)\setminus F$, then $F$ is not a flat.
For such an $x$, there exists a circuit $C$ such that $x\in C\subseteq F\cup \{x\}$.
As $C$ is a circuit, the intersection $F\cap C$ is an independent set in $F$, and $C = (F\cap C) \cup \{x\}$ is dependent, so $F$ cannot be a flat.

Conversely, assume that $\cl(F) = F$ and $(M,I(M))$ is a finitary matroid.
Suppose there exist $x\in M\setminus F$ and an independent set $\sigma\subseteq F$ such that $\sigma\cup \{x\}$ is a dependent set.
Then, by \ref{lm:finitaryAndCircuits}, $\sigma\cup \{x\}$ contains a circuit $C$ which must contain $x$.
But this implies that $x\in \cl(F) = F$, a contradiction.
Thus $\sigma\cup \{x\}\in I(M)$, and hence $F$ is a flat.
\end{proof}

Next, we will focus on relating flats to the closure of sets.
We start out with finite independent sets.
Recall that $I(M)_{\fin}$ denotes the collection of finite independent sets of $(M,I(M))$.

\begin{lemma}
\label{lm:finiteIndFlatProperties}
Let $(M,I(M))$ be a pi-space.
If $\sigma\in I(M)$ is finite, then $\rk(\cl(\sigma)) = |\sigma|$, and $\cl(\sigma)$ is a flat.
\end{lemma}

\begin{proof}
Let $\sigma\in I(M)$ be a finite independent set.
We prove first that $\rk(\cl(\sigma)) = |\sigma|$.
Suppose $\tau\subseteq \cl(\sigma)$ is a finite independent set of cardinality strictly larger than $\sigma$.
By (I2), we can suppose that $\tau\supsetneq \sigma$.
Thus we have $x\in \tau\setminus \sigma$ such that $\sigma \cup \{x\}$ is independent.
But by definition there is a circuit $C$ such that $x\in C\subseteq \sigma\cup \{x\}$, a contradiction.
This shows that $\rk(\cl(\sigma)) = |\sigma|$.

Next, we prove that $\cl(\sigma)$ is a flat.
Suppose $\tau\subseteq \cl(\sigma)$ is independent and let $x\in M\setminus \cl(\sigma)$.
We must show that $\tau\cup \{x\}$ is independent.
Since $x$ lies outside the closure of $\sigma$, we must have $\sigma\cup\{x\}\in I(M)$.
Also, $|\tau|\leq |\sigma| < |\sigma| + 1$ by the first part.
Hence, we can take $\tau'\subseteq (\sigma \setminus \tau) \cup \{x\}$ such that $|\tau\cup\tau'| = |\sigma|+1$.
Note that we cannot have $\tau'\subseteq \sigma$, otherwise $\tau\cup\tau'$ would be independent in $\cl(\sigma)$ of cardinality strictly larger than $|\sigma|$.
Therefore, $x\in \tau'$, so $\tau\cup\{x\}$ is independent.
\end{proof}

\begin{lemma}
\label{lm:flatContainsClosure}
Let $(M,I(M))$ be a pi-space, and $F\subseteq M$ a flat.
\begin{enumerate}
    \item If $X\subseteq F$, then $F$ contains $\cl(X)$.
    \item If $\sigma\in I(M,F)$ and $\rk_{I(M)}(F) = |\sigma|<\infty$, then $F = \cl(\sigma)$.
\end{enumerate}
\end{lemma}

\begin{proof}
Item (1) follows from \ref{lm:flatClosure} since $\cl_{I(M)}(X)\subseteq \cl_{I(M)}(F) = F$.

For item (2), suppose that $F$ is a flat of finite rank, and $\sigma\in I(M,F)$ has rank $|\sigma| = \rk_{I(M)}(F)$.
As $\cl_{I(M)}(\sigma)$ is a flat by \ref{lm:finiteIndFlatProperties}, if $x\in F\setminus \cl_{I(M)}(\sigma)$, then $\sigma\cup \{x\}$ is an independent set contained in $F$ with rank strictly larger than the rank of $F$ given by $|\sigma| = \rk_{I(M)}(F)$, a contradiction.
Thus, $F = \cl_{I(M)}(\sigma)$.
\end{proof}

\begin{corollary}
\label{coro:finiteRankFlats}
Let $(M,I(M))$ be a pi-space.
Then finite rank flats are exactly the sets $\cl(\sigma)$, with $\sigma\in I(M)$ finite.
\end{corollary}

\begin{proof}
This follows from \ref{lm:finiteIndFlatProperties} and \ref{lm:flatContainsClosure}.
\end{proof}

Indeed, we can map any independent set $\sigma$ to the smallest flat containing it.

\begin{corollary}
\label{coro:PhiMap}
Let $(M,I(M))$ be a pi-space and $\sigma\in I(M)$.
Then there exists a unique minimal flat $\Phi(\sigma)$ containing $\sigma$ (and hence $\cl(\sigma)$), and $\Phi:I(M)\to \L(M,I(M))$ gives rise to an order-preserving map.
Moreover, $\rk(\Phi(\sigma))\geq |\sigma|$, and if $\sigma$ is finite then $\Phi(\sigma) = \cl(\sigma)$, which has rank $|\sigma|$.
\end{corollary}

\begin{proof}
Let $X_{\sigma}$ be the set of flats of $M$ containing $\sigma$.
Since $M\in X_{\sigma}$, this is a non-empty set.
Thus, by \ref{lm:intersectionFlats}, $F_{\sigma} = \bigcap_{F \in X_{\sigma}} F$ is the minimal flat containing $\sigma$.
Thus $\Phi(\sigma) = F_{\sigma}$, which clearly defines an order-preserving map.
By \ref{lm:flatContainsClosure}, $\cl(\sigma)\subseteq \Phi(\sigma)$.
It is also clear that $\rk(\Phi(\sigma)) \geq |\sigma|$.
The second part of the Moreover part follows from \ref{lm:finiteIndFlatProperties}.
\end{proof}

Now we can prove the first part of \ref{thm:homotopyPosetOfFlats}.

\begin{theorem}
\label{thm:homotopyTypeFlatsFiniteRank}
Let $(M,I(M))$ be a pi-space of finite rank.
Then $\L(M,I(M))$ is Cohen-Macaulay of dimension $\rk_{I(M)}(M)$.
\end{theorem}

\begin{proof}
We must check that every interval is spherical.
For proper intervals, this follows by induction on the rank and \ref{lm:lower_interval_flats} and \ref{lm:upper_interval_flats}.
Thus, it remains to prove that the proper part of $\L(M) := \L(M,I(M))$ is spherical.

Let $n$ be the rank of $M$ and $\red{\L(M)} = \L(M) \setminus \{0_{I(M)}, M\}$.
Consider the $(n-2)$-skeleton of $I(M)$ and the ``spanning map" $\Phi:I(M)^{(n-2)} \to \red{\L(M)}$ that sends an independent set $\sigma$ to the smallest flat containing it (see \ref{coro:PhiMap}).
As $M$ has finite rank, this map is just $\Phi(\sigma) = \cl(\sigma)$, and $\rk(\Phi(\sigma)) = |\sigma|$.
Now $\Phi^{-1}(\red{\L(M)}_{\subseteq F}) = I(M, F)$, which is Cohen-Macaulay of dimension $\rk(F)-1$ by \ref{thm:homotopyIndependentSetsPoset} applied to the pi-space $(F,I(M,F))$.
Also, $\red{\L(M)}_{\supset F}$ is a proper interval, so it is spherical of dimension $n - \rk(F) - 2$ by \ref{lm:upper_interval_flats} and induction.
Hence, $\Phi^{-1}(\red{\L(M)}_{\subseteq F}) * \red{\L(M)}_{\supset F}$ is $(n-3)$-connected, and by Quillen's fiber \ref{thm:quillen}, we conclude that $\Phi$ is an $(n-2)$-equivalence.
Finally, since $I(M)^{(n-2)}$ is $(n-3)$-connected, we conclude that $\red{\L(M)}$ is $(n-3)$-connected and hence spherical of dimension $n-2$.
This concludes with the proof of the theorem.
\end{proof}

Notice that we have used the same map $\Phi$ as described in \cite[Lemma 3.8]{PW} when restricted to ``partial bases" (see \cite[Definition 4.4]{PW}), which are interpreted as independent sets in our context.

For the infinite-rank case of \ref{thm:homotopyPosetOfFlats}, we will need the analogue of \ref{prop:finReduction}.
We denote by $\redm{\L(M,I(M))_{\fin}}$ the subposet of non-zero finite rank flats of $(M,I(M))$.
More generally, if $\P$ is a poset with a unique minimal element $0_{\P}$, then $\redm{\P}$ denotes the subposet $\P\setminus \{ 0_{\P}\}$.
Note $\redm{\L(M,I(M))_{\fin}} = \L(M,I(M)) \setminus \{0_{I(M)}\}$ if $M$ has finite rank.
Since this is a contractible poset (with maximal element $M$), and the homotopy type of its proper part (that is, after removing $M$) is described in \ref{thm:homotopyTypeFlatsFiniteRank}, we will focus on the case where $(M,I(M))$ has infinite rank.
We will show that, in such a case, the full poset $\red{\L(M,I(M))}$ retracts onto $\redm{\L(M,I(M))_{\fin}}$, and the latter is contractible.

We introduce the following useful language:

\begin{definition}
Let $(M,I(M))$ be a pi-space.
A subset $X\subseteq M$ is termed \textit{finitely closed} if $\cl(\sigma)\subseteq X$ for any independent set of finite rank $\sigma$ contained in $X$.
\end{definition}

Note by \ref{lm:flatContainsClosure} that every flat is finitely closed.

The following result will allow us to conclude when certain closures of sets are flats, and that the rank does not change.

\begin{lemma}
\label{lm:finiteRankSets}
Let $(M,I(M))$ be a pi-space.
Then the following hold:
\begin{enumerate}
    \item $\rk_{I(M)}(\cl(X)) = \rk_{I(M)}(X)$ for any $X\subseteq M$.
    \item If $X\subseteq M$ has finite rank, then $\cl(X)$ is a flat, and so $\cl(\cl(X)) = \cl(X) = \cl(\sigma)$ for any $\sigma\in I(M,X)$ of rank $\rk_{I(M)}(X)$.
    In particular, if $X$ is finitely closed of finite rank, then it is a flat.
    \item If every dependent set of $M$ contains a circuit, then $\cl(X)$ is a flat for any $X\subseteq M$.
\end{enumerate}
\end{lemma}

\begin{proof}
Let $X\subseteq M$ be arbitrary, and let $x\in M\setminus \cl(X)$ and $\sigma\in I(M,X)$ an independent set.
Then $\sigma \cup \{x\}$ cannot contain a circuit, as otherwise this would imply that $x\in \cl(X)$.
From this, it is not hard to see that $\cl(X)$ is a flat if either $X$ has finite rank or every dependent set of $M$ contains a circuit (so item (3) holds).
Note that this, together with item (1) and \ref{lm:flatContainsClosure}, also proves item (2).

Finally, for item (1), note that we always have $\rk_{I(M)}(\cl(X)) \geq \rk_{I(M)}(X)$, with equality if $X$ has infinite rank.
Hence, assume that $X$ has finite rank, say $m$, and that $\cl(X)$ contains an independent set $\sigma$ of rank $m+1$.
By taking an independent set in $X$ of maximal rank $m$ and invoking the exchange property, we can suppose that $\sigma\cap X$ has size $m$.
But then we have that $\sigma\setminus X = \{x\}$, that is, $x\in \cl(X)\setminus X$.
Now, there is a circuit $C$ such that $x\in C \subseteq X \cup \{x\}$.
Thus, $\tau = C\cap X$ has size at most $m$, and by the exchange property there is $\tau'\subseteq \sigma\setminus \tau$ such that $\tau\cup \tau'$ is an independent set of size $m+1$.
But then $x\in \tau'$, that is, $C \subseteq \tau\cup \tau'$ is an independent set, a contradiction.
This proves that $\cl(X)$ cannot contain independent sets of size $\geq m+1$, so item (1) holds.
\end{proof}

The following theorem contains the final ingredients to prove the infinite-rank case of \ref{thm:homotopyPosetOfFlats}.

\begin{theorem}
\label{thm:contractibleFlats}
Let $(M,I(M))$ be a pi-space.
\begin{enumerate}
    \item If $F,G \in \L(M,I(M))$ have finite rank, then $F\vee G$ is a finite rank flat.
    Moreover, if $\sigma\in I(M,F)$ and $\tau\in I(M,G)$ realize the rank of $F$ and $G$ respectively, then $F\vee G = \cl(\sigma\cup \tau)$.
    \item If $X \subseteq M$ is finitely closed of positive rank, then the subposet $\redm{\L(M,I(M))_{\fin \subseteq X}}$ of non-zero finite-rank flats contained in $X$ is contractible.
    \item Suppose that $\P$ is a subposet of $2^M$ such that:
    \begin{enumerate}
        \item $\{X\in \P \tq \rk_{I(M)}(X) < \infty\} =  \redm{\L(M,I(M))_{\fin}}$,
        \item Every element of $\P$ is finitely closed.
    \end{enumerate}
    Then $i: \redm{\L(M,I(M))_{\fin}} \hookrightarrow \P$ is a homotopy equivalence.
\end{enumerate}
In particular, $\redm{\L(M,I(M))_{\fin}}$ and $\P$ are contractible.
\end{theorem}

\begin{proof}
We first prove (1).
Suppose $F,G$ are flats of finite rank.
Recall from \ref{lm:intersectionFlats} that $F\vee G$ exists since $\L(M,I(M))$ is a lattice.
Let $\sigma,\tau\in I(M)$ be flats of maximal rank of $F$ and $G$ respectively.
Then, by \ref{lm:flatContainsClosure}, $F = \cl(\sigma)$ and $G = \cl(\tau)$.
Thus, $F,G\subseteq \cl(\sigma\cup \tau)$.
Since $\sigma\cup \tau$ is a finite set, it has finite rank, and therefore its closure is a flat of finite rank by \ref{lm:finiteRankSets}.
This implies that $F\vee G\subseteq \cl(\sigma\cup\tau)$, and $F\vee G$ has finite rank.

Next, we prove that $F\vee G$ contains $\cl(\sigma\cup\tau)$.
Suppose that $x\in \cl(\sigma\cup\tau) \setminus (F\vee G)$.
Then $x\notin \sigma\cup\tau$, so there is a circuit $C$ such that $x\in C\subseteq \sigma\cup\tau\cup \{x\}$.
Also, $C\cap (\sigma\cup \tau)$ is an independent set contained in $F\vee G$.
As $F\vee G$ is a flat and $x\notin F\vee G$, $C = C\cap (\sigma\cup \tau) \cup \{x\}$ must be an independent set, a contradiction.
Therefore, we must have $F\vee G \supseteq \cl(\sigma\cup \tau)$.
Together with the previous paragraph, this concludes with the proof of item (1).

Now we prove item (2). Let $X\subseteq M$ be a finitely closed subset of positive rank.
This means that $\cl(\sigma)\subseteq X$ for every finite independent set $\sigma\subseteq X$.
By \ref{coro:finiteRankFlats},
\[ \redm{\L(M,I(M))_{\fin \subseteq X}} = \{ \cl(\sigma) \tq \sigma\in I(M,X) \text{ is finite and non-empty}\}.\]
Since $X$ has positive rank, the previous poset contains a flat $G$ of positive rank.
By item (1), if $F\in \redm{\L(M,I(M))_{\fin \subseteq X}}$ then $F\vee G$ is a flat of finite rank which is equal to $\cl(\sigma\cup \tau)$ for $\sigma,\tau \in I(M,X)$ realizing the rank of $F$ and $G$ respectively.
By item (1) in \ref{lm:finiteRankSets}, there is an independent set $\sigma'\subseteq \sigma\cup \tau$ whose rank is $\rk(\sigma\cup \tau) = \rk(\cl(\sigma\cup\tau)) = \rk(F\vee G)$.
Since $X$ is finitely closed, $\cl(\sigma')\subseteq X$, and $\cl(\sigma') = F\vee G$ by \ref{lm:flatContainsClosure}.
Therefore, $F\vee G\subseteq X$.
We have shown then that we have a homotopy
\[ F \subseteq F\vee G \supseteq G\]
whose terms lie in $\redm{\L(M,I(M))_{\fin \subseteq X}}$.
That is, $\redm{\L(M,I(M))_{\fin \subseteq X}}$ is contractible.
This proves item (2) and that $\redm{\L(M,I(M))_{\fin}}$ is contractible by taking $X = M$.

Next, to show item (3), we invoke Quillen's fiber \ref{thm:quillen} for the inclusion $i$.
If $X\in \P$, then $X$ is finitely closed of positive rank, and
\[ i^{-1}( \P_{\subseteq X} ) = \redm{\L(M,I(M))_{\fin \subseteq X}}\]
is contractible by item (2).
By Quillen's fiber theorem, $i$ is a homotopy equivalence.
\end{proof}

\begin{proof}[Proof of \ref{thm:homotopyPosetOfFlats}]
We have that $\L(M,I(M))$ is a lattice by \ref{lm:intersectionFlats}, and it is Cohen-Macaulay when $M$ has finite rank by \ref{thm:homotopyTypeFlatsFiniteRank}.
If $M$ has infinite rank, then $\P = \red{\L(M,I(M))}$ is contractible by \ref{thm:contractibleFlats}.
\end{proof}



We close this section with a relation between the lattice of flats of the associated pi-spaces $(M,I(M)_{\fin})$ and $(M,I(M)_W)$ of a given pi-space $(M,I(M))$.

\begin{corollary}
\label{coro:inclusionsFlats}
Let $(M,I(M))$ be a pi-space.
Then we have inclusions
\[ i: \L(M,I(M)_{\fin}) \hookrightarrow \L(M,I(M))\ \text{ and } \ j: \L(M,I(M)) \hookrightarrow \L(M,I(M)_W),\]
which induce homotopy equivalences between their proper parts.
\end{corollary}

\begin{proof}
It is easy to see that the inclusion $\L(M,I(M)_{\fin}) \subseteq \L(M,I(M))$ holds.
We show that $\L(M,I(M)) \subseteq \L(M,I(M)_W)$.
Let $F\in \L(M,I(M))$, $x\in M\setminus F$, and $\sigma \in I(M)_W$ such that $\sigma\subseteq F$.
If $\sigma$ is finite, then $\sigma\in I(M,F)$ and so $\sigma\cup \{x\}\in I(M)\subseteq I(M)_W$ is independent.
If $\sigma$ is infinite, then for any finite subset $\tau$ of $\sigma$ we have $\tau\cup \{x\}\in I(M) \subseteq I(M)_W$.
Thus $\sigma\cup \{x\}\in I(M)_W$ is independent.
This proves that $F$ is a flat of $(M,I(M)_W)$, concluding with the proof of the second inclusion.

Finally, if $(M,I(M))$ has finite rank, then $I(M) = I(M)_{\fin} = I(M)_W$ and the three posets of flats are equal (and the inclusions are the identity map).
If $(M,I(M))$ has infinite rank, then so do $(M,I(M)_{\fin})$ and $(M,I(M)_W)$.
Hence, in that case, the three posets of flats are contractible, and in particular the inclusions are homotopy equivalences.
\end{proof}

\begin{example}
    Consider the setting of \ref{ex:graphicmatroid} and let $G$ be a 
    $1$-dimensional CW-complex or, equivalently, an undirected graph with edge set $M=E(G)$ and vertex set $V(G)$. To avoid technicalities 
    we assume $G$ does not contain isolated vertices and loops.
    
    From \ref{ex:graphicmatroid}, we know that the elements of $I(M)_{\fin}$ are the compact elements of $I(M)$ 
    (equivalently the finite forests in 
    $E(G)$) and that  
    $I(M)_W = I(M)$. 
    
    The flats of $M$ are in bijection to the 
    (possibly infinite) partitions of $V(G)$ for which the restriction
    of $M$ to any block is connected. 
    
    Let $G = K_{V(G)}$ be the complete graph on $V(G)$; that is
    the CW-complex such that any two vertices are connected by
    exactly one $1$-cell. Then 
    $\L(M,I(M))$ can be identified with the infinite partitions 
    lattice on $V(G)$
    and coincides with $\L(M,I(M)_{\fin})$.
\end{example}

\section{Posets of closed sets} 
\label{sec:closedSets}

In general, the notions of flat and closed set do not coincide.
Therefore, it makes sense to consider posets of closed sets 
separately.

\begin{definition}
\label{def:closedSubsets}
Let $(M,I(M))$ be a pi-space.
A subset $Y\subseteq M$ is \textit{closed} if $Y = \cl(Y)$.
The poset of closed subsets is the collection
\[ \C(M,I(M)) := \{ Y\subseteq M \tq \cl(Y) = Y\},\]
with ordering induced by inclusion.
\end{definition}

Note that $M$ itself is a closed set, and the set of loops $0_{I(M)} = \{x\in M\tq \{x\}\notin I(M)\}$ is also closed.
Moreover, if $x\in 0_{I(M)}$ and $X\subseteq M$ then $x\in \cl(X)$ as $\{x\}$ is a circuit.
This shows that $\C(M,I(M))$ is a bounded poset whose unique minimal element is $0_{I(M)}$, and its unique maximal element is $M$.

On the other hand, by \ref{lm:flatClosure}, every flat is closed, so we have an inclusion of posets $\L(M,I(M)) \hookrightarrow \C(M,I(M))$.
Moreover, if $X\subseteq M$ has finite rank, then by \ref{lm:finiteRankSets}, $\cl(X)$ is a flat of finite rank, so $\cl(\cl(X)) = \cl(X)$.
Then, if we write $\C(M,I(M))_{\fin}$ for the closed sets of finite rank of $(M,I(M))$, we have
\[ \L(M,I(M))_{\fin} = \C(M,I(M))_{\fin}.\]
Also note by \ref{lm:finiteRankSets}(3) that if every dependent set contains a circuit, then $\cl(X)$ is always a flat, and hence $\L(M,I(M)) = \C(M,I(M))$.
By \ref{lm:finitaryAndCircuits}, finitary matroids satisfy this property.
In particular, if $M$ has finite rank, then $\L(M,I(M)) = \C(M,I(M))$.

Now suppose that $M$ has infinite rank, and take $\P := \red{\C(M,I(M))}$.
If $Y \in \P$ and $\sigma\in I(M,Y)$, then $\cl(\sigma)\subseteq \cl(Y) = Y$, so $Y$ is finitely closed of positive rank.
This shows that $\P$ satisfies the hypotheses of \ref{thm:contractibleFlats}(3), and hence it is contractible.

Finally, the inclusion $\red{\L(M,I(M))} \hookrightarrow \red{\C(M,I(M))}$ is a homotopy equivalence since it is the identity map if $(M,I(M))$ has finite rank, or both posets are contractible otherwise.

The following theorem summarizes these observations on $\C(M,I(M))$:

\begin{theorem}
\label{thm:homotopyTypeClosedSets}
Let $(M,I(M))$ be a pi-space.
Then the following hold:
\begin{enumerate}
    \item $\L(M,I(M)) \subseteq \C(M,I(M))$, with equality if every dependent set contains a circuit.
    \item $\L(M,I(M))_{\fin} = \C(M,I(M))_{\fin}$.
    \item If $M$ has finite rank, then $\L(M,I(M)) = \C(M,I(M))$ is Cohen-Macaulay.
    \item If $M$ has infinite rank, then $\red{\C(M,I(M))}$ is contractible.
    \item The inclusion $\red{\L(M,I(M))} \hookrightarrow \red{\C(M,I(M))}$ is a homotopy equivalence.
\end{enumerate}
\end{theorem}

The following example shows that the converse on the equality of item (1) above might not hold.

\begin{example}
    Let $M = \{1,2,3,\ldots\}$ and let $I(M)$ be the set of
    all subsets $A$ of $M$ for which there is no number $n$ such that
    $A$ contains all multiples of $n$.
    We claim that there are no circuits in $M$.
    Indeed, if $X$ is a dependent set, then it contains all multiples of a number $n$. But then $Z$ also
    contains all multiples of $2n$, which is a proper dependent 
    subset. Hence, $X$ cannot be a circuit.
    This shows that $\cl(X) = X$ for all $X \subseteq M$. 

    Similarly, we see that every subset of $M$ is a flat.
    Let $X\subseteq M$, $x\in M\setminus X$, and $\sigma\in I(M,X)$.
    Suppose that $\sigma\cup\{x\}$ is a dependent set, so it contains all multiples of a given number $n$.
    As $\sigma$ is independent, all such multiples are attained once we have added $x$.
    This means that $x$ is a multiple of $n$. But then, $\sigma$ must contain all multiples of $2x$, contradicting that $\sigma$ is independent.
    Therefore, $\sigma \cup \{x\}$ must be independent.

    This example shows that $\L(M,I(M)) = \C(M,I(M))$, and $(M,I(M))$ contains infinite descending chains of dependent sets and no circuits.
    Therefore, the converse of \ref{thm:homotopyTypeClosedSets}(1) does not hold.
\end{example}

We get an analogue of \ref{coro:inclusionsFlats}, with a slight variation on the inclusions.

\begin{corollary}
\label{coro:inclusionsClosedSets}
Let $(M,I(M))$ be a pi-space.
Then we have inclusions
\[ k: \C(M,I(M)_{\fin}) \hookrightarrow \C(M,I(M)_W)\ \text{ and } \ j: \C(M,I(M)) \hookrightarrow \C(M,I(M)_W),\]
which induce homotopy equivalences between their proper parts.
\end{corollary}

\begin{proof}
Recall that $(M,I(M)_W)$ is a finitary matroid, so its circuits are finite.
In particular, a circuit of $(M,I(M)_W)$ is a circuit of $(M,I(M)_{\fin})$ and of $(M,I(M))$.
Therefore, if $Y\in \C(M,I(M)_{\fin})$ or $Y\in \C(M,I(M))$, then $\cl_{I(M)_W}(Y) = Y$. 
This proves the inclusions $k$ and $j$.
The homotopy equivalences follow easily as in the proof of \ref{coro:inclusionsFlats}: if $M$ has finite rank, then the three posets are the same, and hence we have homotopy equivalences between their proper parts.
Otherwise, the three matroids $(M,I(M))$, $(M,I(M)_{\fin})$ and $(M,I(M)_W)$ have infinite rank and hence their corresponding posets of closed subsets have a contractible proper part by \ref{thm:homotopyTypeClosedSets}.
\end{proof}

The following example shows that, in contrast to \ref{coro:inclusionsFlats}, we might not have an inclusion $\C(M,I(M)_{\fin}) \subseteq\C(M,I(M))$.

\begin{example}
Let $(M,I(M))$ be a pi-space such that $M$ is infinite, $I(M)$ contains every finite subset of $M$, and $\C(M,I(M)) \neq 2^M$.
Then $(M,I(M))$ must have circuits, but $(M,I(M)_{\fin})$ has no circuits since its dependent sets are exactly the infinite subsets of $M$ (see \ref{ex:finiteSubsetsIndependent}).
In particular, $\C(M,I(M)_{\fin}) = 2^M \not\subseteq \C(M,I(M))$.

For example, if $I(M)$ is the set of proper subsets of $M$, then $\C(M,I(M))$ equals all subsets of $M$ except for those of the form $M\setminus \{x\}$, with $x\in M$.
\end{example}

In classical matroid theory, we have that a subset $F$ of the underlying set $M$ of a matroid is a flat if and only if it is closed.
Moreover, the closure of any subset of $M$ is a flat.
However, in our more general context, this property is not immediate.
This observation and \ref{thm:homotopyTypeClosedSets} motivate the following questions.

\begin{question}
\label{question:flatsVsClosedSets}
Let $(M,I(M))$ be a pi-space, and let $X\subseteq M$.
\begin{enumerate}
    \item Is $\cl(X)$ a closed set?
    \item Is $\cl(X)$ a flat?
    \item Is $\C(M,I(M))$ a (complete) lattice?
    \item When do we have $\C(M,I(M)) = \L(M,I(M))$?
\end{enumerate}
\end{question}

Note that if $\cl(X)$ is a flat, then $\cl(\cl(X)) = \cl(X)$ by \ref{lm:flatClosure}, so it is closed.
Therefore, if (2) above holds, then (1), (3), and (4) also hold.
It follows for example by \ref{lm:finiteRankSets}(3), that $\cl(X)$ is a flat for every subset $X$ if every dependent set contains a circuit.
The latter property holds for finitary matroids, so all items in \ref{question:flatsVsClosedSets} have a positive answer in that case.

The following example shows that closed sets are not always flats, and in general $\C(M,I(M)) \neq \L(M,I(M))$.

\begin{example}
Let $M$ be a vector space of dimension at least two over an infinite field, equipped with a non-degenerate orthogonal form.
For instance, we can take $M$ to be an Euclidean space of dimension at least two with the classical inner product.
Let $I(M)$ consist of sets of vectors of $M$ that contain only finitely many pairs of orthogonal vectors.
Clearly, $I(M)$ contains every finite set of vectors of $M$.
Hence, $(M,I(M))$ is a pi-space.

Now, a dependent set is a set of vectors $S$ of $M$ containing an infinite number of pairs of orthogonal vectors.
Then, either there is a vector $v\in S$ that is orthogonal to an infinite number of vectors in $S$, or every vector in $S$ is orthogonal to only a finite number of vectors in $S$.
In the former case, we can take any $w\in S \setminus \{v\}$ orthogonal to $v$, and then $S\setminus \{w\}$ is a dependent set.
In the latter case, removing any vector from $S$ yields again a dependent set.
Therefore, $S$ contains no circuits, and, in particular, $(M,I(M))$ has no circuits.
This shows that every subset of $M$ is closed.

On the other hand, consider a vector $v\in M$ that is not isotropic (i.e., it is not orthogonal to itself), and let $X$ be the subspace generated by $v$ after removing the zero vector.
Then $X$ is an independent set.
As $M$ has dimension at least two, there is a non-zero orthogonal vector $w$ to $v$.
Since $X\cup \{w\}$ has an infinite number of pairs of orthogonal vectors, we see that $X$ is not a flat.
Thus, $X\in \C(M,I(M)) \setminus \L(M,I(M))$.
\end{example}

We close this section with an example of a pi-space with an infinite chain of proper inclusions of a set into its closure.

\begin{example}
\label{ex:infiniteAscChainNonClosed}
Let $M$ be the set of natural numbers, and let $(a_n)_{n\geq 1}$ and $(b_n)_{n\geq 1}$ be two sequences of pairwise-distinct natural numbers such that $a_n\neq b_m$ for all $n,m$, and such that $M\setminus \{a_n,b_n \tq n\geq 1\}$ is infinite.
Write $P_n = \{a_m \tq m > n\} \cup \{b_1,\ldots,b_n\}$.
Therefore, the $P_n$ are infinite sets, pairwise incomparable (that is, $\{P_n \tq n\geq 1\}$ is an antichain in the poset of subsets of $M$).
Let $I(M)$ consist of all subsets of $M$ which do not contain any $P_n$.
As $I(M)$ contains all finite subsets of $M$, $(M,I(M))$ is a pi-space.
Finally, we note that the dependent sets of $M$ are exactly the subsets of $M$
containing some $P_n$. Since the $P_n$ are pairwise incomparable, they are also the circuits.

Now, let $X_1 = \{a_n \tq n\geq 2\}$, and $X_{n+1} = \cl(X_{n})$.
We claim that
\[ X_{n+1} = \{a_m \tq m\geq 2\} \cup \{b_1,\ldots,b_{n}\} = X_{n}\cup \{b_n\}.\]
Let $x\in X_{n+1} \setminus X_n$.
Then there is a circuit $C$ such that $x\in C \subseteq X_n\cup \{x\}$.
Thus, $C = P_m$ for some $m\geq 1$.
As $a_k\in X_n$ for all $k\geq 2$ and $a_1\notin P_m$, we see that $x\in \{b_1,\ldots, b_m\}$.
Hence, by induction on the description of $X_n$, we must have $\{ b_1,\ldots,b_{m}\} \subseteq \{b_1,\ldots,b_{n-1}\} \cup \{x\}$,
From this, we see that $m = n$ and $x = b_n$.
As this argument also shows that indeed $b_n\in X_{n+1}$, we conclude that $X_{n+1} = X_n\cup \{b_n\}$.

We have produced an infinite chain of proper inclusions $X_n\subsetneq X_{n+1}$ with $X_{n+1} = \cl(X_n)$.
This example then shows that the closure of a subset of a pi-space might not be closed, and that it might even give rise to an infinite ascending chain of non-closed sets.
In particular, no element of this chain can be a flat (cf. \ref{lm:flatClosure}).
\end{example}

\appendix

\section{Homotopy Tools}
\label{app:homotopy}

Our main tool for homotopy equivalences will be Quillen's fiber theorem.
In this appendix, we state the theorem and recall the necessary definitions from topology and combinatorics.

Let $m$ be an integer, and $X$ a topological space.
Then $X$ is said to be $m$-connected if $\pi_n(X,x_0) = 1$ for all $n < m$ and $x_0\in X$.
By convention, every topological space is $m$-connected if $m\leq -2$, and a space is $(-1)$-connected if and only if it is non-empty.
A continuous map between topological spaces $f:X\to Y$ is an $m$-equivalence if $f$ induces an isomorphism between homotopy groups of degree $<m$, and an epimorphism in degree $m$.
If $K$ is a finite-dimensional simplicial complex of finite dimension $\dim K$, then $K$ is said to be spherical if it is $(\dim K - 1)$-connected.
Equivalently, $K$ has the homotopy type of a wedge of spheres of dimension $\dim K$.
We say that $K$ is (homotopically) Cohen-Macaulay if for every simplex $\sigma \in K$, the link $\Lk_K(\sigma)$ is spherical of dimension $\dim K - |\sigma|$.
Note that, in particular, if $K$ is Cohen-Macaulay, then it is spherical by taking $\sigma$ the empty simplex.

If $Q$ is a poset and $x \in Q$, then we write $Q_{\leq x}$ for the subposet
of all $x' \in Q$ with $x' \leq x$.
If $P\subseteq Q$, we will also write $P_{\leq x}$ even if $x\in Q\setminus P$.
Analogously defined are 
$Q_{< x}$, $Q_{\geq x}$ and $Q_{> x}$.
For two posets $P$ and $Q$ on disjoint sets we write $P*Q$ for the 
poset on $P \cup Q$ which inherits the order relations among
elements of $P$ and among elements of $Q$, and sets $y < x$ for all
$y \in P$ and $x \in Q$. 
The height of $Q$ is the 
supremum of the numbers $n$ for which there is a 
strictly ascending chain $x_0 < \cdots < x_n$ in $Q$. Finally,
we always consider a poset both as a combinatorial object and, 
via its order complex, as a topological space.
We denote by $|Q|$ the geometric realization of a poset $Q$.
In particular, order-preserving maps between posets induce continuous maps between their geometric realizations.
If $f,g:P\to Q$ are order-preserving maps and $f(y)\leq g(y)$ for all $y\in P$, then $f,g$ give rise to homotopic maps.
Finally, $|P*Q| \cong |P|*|Q|$, and, in general, the join of an $m$-connected space with an $n$-connected space is $(m+n+2)$-connected.

\begin{theorem}
[Quillen's fiber theorem]
\label{thm:quillen}
Let $P,Q$ be two posets, and $f:P\to Q$ an order-preserving map.
\begin{enumerate}
   \item Let $m\geq -1$, and suppose $P,Q$ have finite height. If $f^{-1}(Q_{\leq x}) * Q_{>x}$ is $m$-connected for all $x\in X$, then $f$ is an $(m+1)$-equivalence.
    \item If $f^{-1}(Q_{\leq x})$ is contractible for all $x\in Q$, then $f$ is a homotopy equivalence.
\end{enumerate}
\end{theorem}

Assertion (2) of \ref{thm:quillen} is an immediate consequence of
(1) (see also \cite{Qui78}). A proof of (1) can be found in \cite[Theorem A.1]{PS} for the finite case, but it easily extends to finite-height infinite posets.

\end{document}